\documentclass[12pt]{amsart}

\usepackage{amsmath} 
\usepackage{amsfonts}
\usepackage{amssymb}
\usepackage{amsthm}

\usepackage{latexsym}
\usepackage{color}
\usepackage{enumerate}
\usepackage{mathrsfs}
\usepackage{comment}
\usepackage{setspace}
\usepackage{ulem}

\usepackage{appendix}

\newtheorem{definition}{Definition}[section]

\newtheorem{thm}{Theorem}[section]
\newtheorem{prop}{Proposition}[section]
\newtheorem{cor}{Corollary}[section]

\theoremstyle{remark}
\newtheorem{remark}{Remark}[section]

\numberwithin{equation}{section}

\def\R{\mathbb{R}}

\def\R{\mathbb{R}}

\def\p{\partial}

\def\m{\mathfrak{m}}

\def\({\left(}
\def\){\right)}
\def\[{\left[}
\def\]{\right]}
\def\S{\Sigma}

\def\hM{\hat{M}}
\def\thM{\widetilde{M}}

\def\Sh{\Sigma_{_H}}

\def\gm{\hat{h}}
\def\th{\tilde{h}}
\def\tgm{\tilde{h}}
\def\bOm{\bar{\Omega}}

\newcommand{\be}{\begin{equation}}
\newcommand{\ee}{\end{equation}}
\newcommand{\bee}{\begin{equation*}}
\newcommand{\eee}{\end{equation*}}

\begin{document}

\title{On a Penrose-like inequality in dimensions less than eight}
   
\author[McCormick]{Stephen McCormick}
\address{Institutionen f\"{o}r Matematik, Kungliga Tekniska H\"{o}gskolan, 100 44 Stockholm, Sweden; and School of Science and Technology, University of New England, Armidale, NSW 2351, Australia.}
\email{stephen.mccormick@une.edu.au}

\author[Miao]{Pengzi Miao}
\address{Department of Mathematics, University of Miami, Coral Gables, FL 33146, USA.}
\email{pengzim@math.miami.edu}

\thanks{The first named author is grateful for support from the Knut and Alice Wallenberg Foundation.\\The second named author's research was partially supported by the Simons Foundation Collaboration Grant for Mathematicians \#281105.}


	
\begin{abstract}
On an asymptotically flat manifold $M^n$ with nonnegative scalar curvature,
with outer minimizing boundary $\Sigma$, we prove a Penrose-like inequality in 
dimensions $ n < 8$, under suitable assumptions on the mean curvature 
and the scalar curvature of $ \Sigma$.
\end{abstract}


\maketitle 
 
\section{Introduction and statement of results} \label{sec-intro}

The Riemannian Penrose inequality is a fundamental inequality in mathematical general relativity. 
It gives  a lower bound for the total mass of an asymptotically flat manifold  $M$ with nonnegative scalar curvature, 
in terms of the area  of the boundary $ \p M$, 
provided $ \p M$ is a minimal hypersurface that is outer minimizing in $M$.  
In this paper, we prove a ``Penrose-like'' inequality in the case where $\p M$ is not a minimal surface. 
To state precisely, both the Riemannian Penrose inequality and our main theorems, 
we first review some definitions.

\begin{definition}
Let $ n \ge 3 $. A Riemannian manifold $M^n$ is called asymptotically flat (with one end)
if there exists a compact set $K $ such that $ M \setminus K $ is diffeomorphic to 
$ \R^n $ minus a ball such that, 
in the coordinate chart coming from the standard coordinates on $\R^n$, 
the metric $h$ on $M^n$  satisfies 
\be \label{eq-decay-conds}
 h_{ij} = \delta_{ij} + O ( | x |^{-p} ) , \ \p h_{ij}  =  O ( | x |^{ - p - 1}) , \  
\p \p h_{ij}  =  O ( | x |^{ - p - 2 } )
\ee 
for some  $ p > \frac{n-2}{2} $ 
and the scalar curvature $ R_h$ of $h$ satisfies $ R_h = O ( | x |^q ) $ for some $ q > n $.
Here $ \partial $ denotes the standard partial differentiation on $ \R^n$. 
\end{definition}

On an asymptotically flat manifold $ M^n$, the limit
\bee
\m = \lim_{ r \rightarrow \infty} \frac{1}{ 2 (n-1) \omega_{n-1} } \int_{ S_r }
(h_{ij,i} - h_{ii,j} ) \nu^j d \sigma 
\eee
exists and is known as the ADM mass (\cite{ADM}) of $M$. Here $ \omega_{n-1}$ 
is the area of the standard unit $(n-1)$-sphere in $ \R^n$,
$ S_r  = \{ x \ | \ | x | = r \}$,   $ \nu $ is the Euclidean  outward unit normal to $ S_r$, $ d \sigma$ is the Euclidean 
 area element on $ S_r$, and summation is implied over repeated indices. 
Under suitable conditions, it was proved by Bartnik \cite{Bartnik-86} and Chru\'sciel \cite{Chrusciel} independently that $\m$
 is a geometric invariant of $M$; 
in particular, the expression for $\m$ above is independent of coordinates satisfying \eqref{eq-decay-conds}.

\begin{definition}
Given an asymptotically flat manifold  $M$ with boundary $ \Sigma$, one says that 
$ \Sigma $ is  outer minimizing if it minimizes area among all hypersurfaces in $M$ that 
enclose $ \Sigma $.
\end{definition} 

\begin{thm}[Riemannian Penrose inequality in dimensions less than $8$] \label{thm-RPI}
Let $M^n$ be an asymptotically flat manifold with nonnegative scalar curvature, 
with boundary $\Sigma$, where $ 3 \le n \le 7$. 
Suppose $ \Sigma $ is a minimal hypersurface that is outer minimizing in $M$, then
\be  \label{eq-Penrose}
	\m \geq \frac12\( \frac{|\Sigma |}{\omega_{n-1}}  \)^{\frac{n-2}{n-1}},
\ee 
where $ \m $ is the ADM mass of $M$ and $ | \Sigma | $ is  the area of $ \Sigma $.  
Moreover, equality holds if and only if $M$ is isometric to a spatial Schwarzschild manifold outside its horizon. 
\end{thm}

When $ n = 3$,  Huisken and Ilmanen \cite{H-I97, H-I01}  first proved Theorem \ref{thm-RPI} 
for the case that $ \Sigma$  is connected, using inverse mean curvature flow method, and later Bray \cite{Bray01} proved Theorem \ref{thm-RPI} for  the general case  
in which  $ \Sigma$ can have multiple components, 
using a conformal flow of metrics and the Riemannian positive mass theorem \cite{Schoen-Yau79, Witten81}.
For higher dimensions $ n < 8 $,  Bray and Lee  \cite{Bray-Lee} proved  Theorem \ref{thm-RPI} 
using Bray's  conformal flow method from \cite{Bray01}. 

In this paper, we apply Theorem \ref{thm-RPI} to prove a Penrose-like inequality 
for  manifolds  whose boundary is not  a minimal hypersurface. 

\begin{thm} \label{thm-main}
Let $M^n$ be  an asymptotically flat  manifold of dimension  $3\leq n\leq 7$ with nonnegative scalar curvature, 
with connected, outer minimizing boundary $\Sigma$.
Let $ g $ be the induced metric on $ \Sigma $ and $H$ be the mean curvature of $\Sigma$. 
Suppose $g$ and $H$ satisfy 
\be \label{eq-condition-H-R-1}
 \min_\Sigma R_g   >  \frac{n-2}{n-1} \max_\Sigma H^2 , 
 \ee
 where $ R_g$ is  the scalar curvature of $(\Sigma, g)$. 
Then
\be \label{eq-Penrose-like-1}
\m  \ge \frac12 \(\frac{| \Sigma |}{\omega_{n-1}}\)^{\frac{n-2}{n-1}} \left( 1 - 
\frac{n-2}{n-1} \frac{  \max_\Sigma H^2 }{ \min_\Sigma R_g } \right) , 
\ee
where $ \m $ is the ADM mass of $M$ and  $ | \Sigma | $ is the area of $ \Sigma $.
\end{thm}

\begin{remark}
If $M^n$ is a spatial Schwarzschild manifold outside a rotationally symmetric sphere of positive mean curvature, 
then  equality   in \eqref{eq-Penrose-like-1} holds on $M^n$. 
\end{remark}

\begin{remark}
The assumption that $ \Sigma$ is outer minimizing implies  $H \ge 0 $ necessarily. 
\end{remark}

If  the mean curvature $H $  of $ \Sigma = \p M$ is  strictly positive, we have the following related but different result. 

\begin{thm}\label{thm-main-Delta}
Let $M^n$ be  an asymptotically flat  manifold of dimension  $3\leq n\leq 7$ 
with  nonnegative scalar curvature, with connected, outer minimizing boundary $\Sigma$. 
Let $ g $ be the induced metric on $ \Sigma $ and $H$ be the mean curvature of $\Sigma$.
Suppose $ H > 0 $ and 
\be \label{eq-R-delta-H0}
	R_g -2H\Delta H^{-1}-\frac{n-2}{n-1}H^2>0,
\ee 
where  $ R_g$  is the scalar curvature of $(\Sigma, g)$ and 
$ \Delta$ is the Laplace-Beltrami operator on $(\Sigma, g)$. 
Then
\be \label{eq-Penrose-like-theta}
	\m   \geq  \frac{1}{2}\( \frac{|\Sigma|}{\omega_{n-1}} \)^{\frac{n-2}{n-1}}\( 1-\theta \),
\ee 
where $ \theta \in (0,1)$ is the constant given by 
$$
\theta = \frac{n-2}{n-1}\,\max\limits_{\Sigma}\frac{H^2}{R_g -2H\Delta H^{-1}} .
$$
Here $ \m $ is the ADM mass of $ M$ and $ | \Sigma | $ is the area of $ \Sigma$.
\end{thm}

\begin{remark}
Both Theorem \ref{thm-main} and Theorem \ref{thm-main-Delta} are applicable to manifolds whose boundary  
is a constant mean curvature (CMC) hypersurface satisfying 
\bee
	R_g  - \frac{n-2}{n-1}H_o^2 >0,
\eee
where $ H = H_o >0 $  is a  constant.
In this case, both theorems coincide, and \eqref{eq-Penrose-like-1} and \eqref{eq-Penrose-like-theta} become
\bee
\m   \geq  \frac{1}{2}\( \frac{|\Sigma|}{\omega_{n-1}} \)^{\frac{n-2}{n-1}}\( 1-   \frac{n-2}{n-1} 
\, \frac{H_o^2}{ \min_\Sigma R_g }  \).
\eee
\end{remark}

\begin{remark} \label{rem-hawking-inequality}
If  $n=3$, 
the right  side of  \eqref{eq-Penrose-like-theta}  can be compared to the Hawking mass of the $2$-surface $\Sigma$ as follows.
Let $ d \sigma $ be the induced area element on $ \Sigma$.
By \eqref{eq-R-delta-H0},  
\begin{align*}
\frac{1}{16 \pi} \int_\Sigma H^2\, d\sigma & = 
\frac{1}{16 \pi}  \int_\Sigma\frac{H^2}{R_g -2H\Delta H^{-1}}  (R_g -2H\Delta H^{-1})  \,d\sigma\\
& \leq  \frac{1}{16 \pi}  
 \max\limits_\Sigma\frac{H^2}{R_g-2H\Delta H^{-1}}  \int_\Sigma \left(R_g -\frac{2}{H^2} | \nabla H |^2 \right) \,d\sigma \\
& \le \frac{1}{2}   \max\limits_\Sigma\frac{H^2}{(R-2H\Delta H^{-1})} = \theta  ,  
\end{align*}
where the last  inequality follows by the Gauss-Bonnet Theorem. Thus, 
when $\Sigma$ is a $2$-surface, 
\bee
\frac{1}{2}\( \frac{|\Sigma|}{\omega_{2}} \)^{\frac{1}{2}}\( 1-\theta \) = \sqrt{  \frac{|\Sigma|}{16 \pi } }\( 1-\theta \) 
\le \m_{_H} (\Sigma),
\eee
where
\bee
 \m_{_H} (\Sigma) = \sqrt{  \frac{|\Sigma|}{16 \pi } }\( 1- \frac{1}{16 \pi} \int_\Sigma H^2  d \sigma \) 
\eee
is the Hawking quasi-local mass \cite{Hawking} of $\Sigma$ in $M^3$. 
In this case, a conclusion
$ \m \ge  \m_{_H} (\Sigma) $,
which is stronger than \eqref{eq-Penrose-like-theta},
would follow from the weak inverse mean curvature flow argument in \cite{H-I01} 
provided $M^3$ satisfies certain topological assumptions that guarantee
the solution to the weak inverse mean curvature flow starting from $ \Sigma$ remains connected. 
A similar remark also applies to Theorem \ref{thm-main} in dimension $n=3$.
\end{remark}

\begin{remark}
The connectedness assumption of $ \Sigma = \partial  M^n $ in both  Theorems \ref{thm-main} and \ref{thm-main-Delta}
is not essential. It is assumed here only for the  simpleness  of the statement of results. See
Theorems  \ref{thm-main-multi-b} and \ref{thm-main-Delta-multi-b} for the case in which $\Sigma$ has multiple components. 
\end{remark}

We now explain  the idea of the proof of Theorems \ref{thm-main} and \ref{thm-main-Delta}, which is largely inspired by the method of  Mantoulidis and Schoen in \cite{M-S}. 
Given a suitable metric $g $ on the $2$-sphere $ S^2$, Mantoulidis and Schoen 
constructed a collar extension of $(S^2, g)$ (cf. \cite[Lemma 1.3]{M-S}), 
which is a metric $\gamma$ on the cylinder $ T =  [0, 1] \times S^2$ satisfying the conditions: 
$ \gamma$ has nonnegative scalar curvature, the induced metric 
on $ \Sigma_0 = \{ 0 \} \times S^2$ agrees with $ g $, 
$ \Sigma_0 $ is a minimal surface,  the induced metric 
on  $ \Sigma_t = \{ t \} \times S^2 $ gets  transformed  into a round metric as $t $ increases  while 
the area of $\Sigma_t $ expands and the mean curvature of $ \Sigma_t$ becomes positive in a controlled fashion. 
For this reason, we would like to call such an extension $(T, \gamma)$  an ``outer collar extension" of $(S^2, g)$. 
Given such an  outer collar extension  $(T, \gamma)$, Mantoulidis and Schoen smoothly attach a 
suitable spatial Schwarzschild manifold to  $(T, \gamma)$ at  $ \Sigma_1$ to obtain an asymptotically  flat  manifold 
which has  desired geometry  near infinity while having  an outer most horizon boundary that is isometric to $(S^2, g)$.   
Under an assumption that $g$ has positive Gauss curvature, 
a similar outer collar extension of $(S^2, g)$, but  with the minimal surface condition replaced by a CMC condition, 
is given in \cite{M-X}.

In contrast to  the use of an outer collar extension as in the above work, 
we make use of  an ``inner collar extension" in the present paper.
More precisely, given the asymptotically flat manifold $(M, g)$ in 
Theorems \ref{thm-main} and \ref{thm-main-Delta}, we construct  
a metric $\gamma$ on a cylinder $ T = [0, b] \times \Sigma$ for some $b>0$ 
satisfying the conditions: $ \gamma$ has nonnegative scalar curvature, the induced metric on 
$ \Sigma_0 = \{ 0 \} \times \Sigma $ agrees with $ g $, 
the mean curvature of $ \Sigma_0$ in $(T, \gamma)$ with respect to the outward normal 
agrees with (or is greater than) $H$, and the area of  $ \Sigma_s = \{ s \} \times \Sigma $ decreases as $s$ 
increases such that the other end $ \Sigma_b$ becomes a minimal hypersurface 
with controlled area. We then attach  this 
 $(T, \gamma)$ to the given
manifold $M$ along the boundary component  $\Sigma_0 = \Sigma$ to obtain an asymptotically flat 
manifold $\hat M$  with an outer minimizing minimal hypersurface boundary $ \Sigma_b$. The metric on $\hat M$
may not be smooth across $\Sigma$, but the mean curvature conditions on the two sides of $\Sigma$ in $\hat M$
guarantee that the Riemannian Penrose inequality can still be applied to $\hat M$ (cf. \cite{Miao09}), which 
gives the proof of Theorems \ref{thm-main} and Theorems \ref{thm-main-Delta}. 
In particular, the quantities on the right-hand side of \eqref{eq-Penrose-like-1} and \eqref{eq-Penrose-like-theta} 
are simply determined by the area of $\Sigma_b$ in $(T, \gamma)$. 

We note that this idea of constructing an inner collar extension, with one end being a minimal hypersurface, 
to extend the non-minimal boundary of an asymptotically flat manifold is in spirit similar to the idea behind 
Bray's inner mass definition \cite{Bray01}.

This paper is organized as follows.  In Section \ref{sec-collars}, we construct a suitable inner collar 
extension of  the boundary data described by a triple $(\Sigma, g, H)$. 
In Section \ref{sec-proof},  we prove  Theorems \ref{thm-main} and Theorems \ref{thm-main-Delta} 
by attaching the inner collar to the given manifold $M$  and applying the Riemannian Penrose  inequality.
 
\section{An inner collar extension}  \label{sec-collars}

In this section, we use  a triple $(\Sigma^{n-1}, g, H)$ to denote 
a connected, closed manifold $ \Sigma$ of dimension $ n-1$,
 a Riemannian metric $g$  on $ \Sigma$, and  a  positive function  $H  $ on $\Sigma$.
We also let $r_o$ be the area  radius of $(\Sigma, g)$, defined  by
$$ r_o=\( \frac{|\Sigma|_g}{\omega_{n-1}} \)^{\frac{1}{n-1}}.$$

Similar to the outer collar extension constructed  in \cite{M-X}, 
we construct an inner collar extension for $(\Sigma^{n-1}, g, H)$ as follows.
Given any $  m \in \left(0,  \frac{1}{2} r^{n-2}_o  \right)$, consider the 
$n$-dimensional,  spatial Schwarzschild manifold 
\bee
\left(M^S_m, \gamma_m \right)  = \left( (r_m , \infty) \times S^{n-1}, \frac{1}{1 - \frac{2 m}{r^{n-2}} } d r^2 + r^2 g_*  \right),
\eee
where $r_m=(2m)^{1/(n-2)}$ and $ g_*$ denotes the standard metric on $ S^{n-1}$ with volume $ \omega_{n-1}$.
Making a change of variable $  s = \int_{r_m}^r \left( 1 - \frac{2m}{r^{n-2}} \right)^{-\frac12} d r$, 
 we can write the Schwarzschild manifold as
\be \label{part-metric}
\left(M^S_m, \gamma_m \right)  =  \left( [0, \infty) \times S^{n-1}, d s^2 +  u_m^2 (s) g_* \right) , 
\ee
where   $u_m (s) $ satisfies $u_{m}(0)=r_m$,

\be \label{eq-um-p}
u_{m}'(s)=\sqrt{1-\dfrac{2 m }{u_{m}(s)^{n-2}}} \ \ 
\mathrm{and} \ \ 
u_{m}''(s)=(n-2)\dfrac{ m }{u_{m}(s)^{n-1}} . 
\ee 
Since $ r_m < r_o $, there exists $ s_o > 0 $  such that $  u_m (s_o ) = r_o $.
In what follows, we will use the finite region  in $(M^S_m , \gamma_m)$ that is bounded by  the horizon $\{ s = 0 \}$ 
and the  sphere $\{ s = s_o\}$, as a model to construct an  inner collar extension of $(\Sigma^{n-1}, g, H)$. 

For $ s \in [0, s_o ]$, let 
\be \label{eq-vms}
v_m(s) =  u_m ( s_o - s  ) .
\ee
Given any smooth positive function $A(x)$ on $\Sigma^{n-1}$, define the metric
\be 
\gamma_A=A(x)^2ds^2+r_o^{-2}v_m(s)^2g
\ee
on  the product $ \Sigma^{n-1}\times[0,s_o] $. For each $s \in [0, s_o]$, the mean curvature $H(s)$ 
of $ \Sigma_s = \{ s \} \times \Sigma $
 with respect to $ \nu = -\partial_s$ is given by
\begin{align}
H_s(x) &=\frac{n-1}{A(x)v_m(s)}\sqrt{1-\frac{2m}{v_m(s)^{n-2}}}  \label{eq-Hx-defn}
\end{align}
by \eqref{eq-um-p} and \eqref{eq-vms}. 
In particular, at   $s = s_o$, 
\begin{equation} \label{eq-H-0}
\begin{split}
H_{s_o}(x)&=\frac{(n-1)}{A(x) r_m }\sqrt{1-\frac{2m}{r_m^{n-2}}} \\
&=0.
\end{split}
\end{equation}
Now, at $ s = 0 $, if we want to impose  $H_0 (x)  = H(x)$, 
we must choose
\be \label{eq-A-defn}
A(x)=\frac{n-1}{H (x)r_o}\sqrt{1-\frac{2m}{r_o^{n-2}}}.
\ee

With such a choice of $A(x)$, 
using \eqref{eq-um-p} and \eqref{eq-A-defn}, 
one  checks (cf. \cite{M-X}) that the scalar curvature of $\gamma_A$  is given by 
\begin{align}
R_{\gamma_A}  &= r_o^2v_m^{-2} \left(  R_g - (n-1)(n-2)r_o^{-2}A^{-2} -2A^{-1}\Delta A\right)\nonumber\\
&= r_o^2v_m^{-2} \left(  R_g - \frac{n-2}{n-1}H^{2}\( 1-\frac{2m}{r_o^{n-2}} \)^{-1} -2H\Delta H^{-1}\right),\label{eq-RA-bound}
\end{align}
where $R_g$ is the scalar curvature of $g$.

This leads us to the following proposition. 

\begin{prop} \label{prop-inner-collar-2}
	Given a triple $(\Sigma^{n-1}, g, H)$, 
	suppose $g$ and $H$  satisfy 
	\be \label{eq-R-delta-H}
	R_g - \frac{n-2}{n-1}H^{2} -2H\Delta H^{-1} >  0. 
	\ee
	Let $ \theta \in (0,1)  $ be the constant given   by 
	\be  \label{eq-def-theta}
	\theta=\frac{n-2}{n-1}\,\max\limits_{\Sigma}\frac{H^2}{R_g-2H\Delta H^{-1}},
	\ee
	and define a constant $m$ and a function $A(x)$ by 
	\be\label{eq-m-defn}
	m=\frac{r_o^{n-2}}{2}\( 1-\theta \)  \ \  \mathrm{and} \ \ 
	A(x)=\frac{n-1}{H(x)r_o}\sqrt{1-\frac{2m}{r_o^{n-2}}},
	\ee
	respectively. 
	Then the metric
	\be 
	\gamma=A(x)^2ds^2+r_o^{-2}v_m(s)^2g
	\ee
	defined on $T=[0,s_o]\times \S$ satisfies
	\begin{enumerate}[(i)]
		\item $R_\gamma \ge 0$, i.e. $\gamma$ has nonnegative scalar curvature; 
		\item the induced metric on $\S_0 = \{ 0 \} \times \Sigma $ by $\gamma$ is $g$;
		\item the mean curvature   of $\S_0$ with respect to 
		$  - \p_s $  equals  $H$; 
		\item $\Sigma_s  = \{ s \} \times \Sigma $ 
		has positive  mean curvature with respect to $ - \p_s$ for each $ s \in [0, s_o ) $; and
		\item the other  boundary component $ \Sigma_{s_o} $ is a minimal surface 
		whose area  satisfies
		\be \label{eq-area-sigma-h}
		\begin{split}
			\frac{1}{2}\( \frac{|\Sigma_{s_o}|}{\omega_{n-1}} \)^{\frac{n-2}{n-1}}
			=  & \  \frac{1}{2}\( \frac{|\Sigma|}{\omega_{n-1}} \)^{\frac{n-2}{n-1}}\( 1-\theta \). \\
		\end{split}
		\ee
	\end{enumerate}	
\end{prop}

\begin{proof}
	By \eqref{eq-RA-bound}, \eqref{eq-def-theta} and \eqref{eq-m-defn}, 
	it is clear that we have $ R_\gamma\ge 0 $, which proves (i); 
	(ii) is evident from the definition of $\gamma$ and the fact $ v_m (0) = r_o $; 
	By  \eqref{eq-A-defn}, $H_0 =H $ at $ \Sigma_0$ which proves (iii); 
	(iv) follows directly from \eqref{eq-Hx-defn}; 
	The fact $ \Sigma_{s_o}$ is a minimal surface follows from \eqref{eq-H-0}.
	Clearly, 
	\bee
	\begin{split}
		|  \Sigma_{s_o}   | = \omega_{n-1}  v_m (s_o)^{n-1} = \omega_{n-1} (2m)^{(n-1)/(n-2)}
	\end{split}
	\eee
	which implies (v) by \eqref{eq-m-defn}. 
\end{proof}

\begin{remark}
When $n=3$,  i.e., $\Sigma$ is a $2$-surface, 
the inner collar extension  constructed above  provides  a  valid {\it fill-in} 
of the triple  $(\Sigma, g, H)$,  as considered by Jauregui \cite[Definition 3]{Jauregui}. 
\end{remark}

When the mean curvature function $H(x)$ is a constant, the following is a direct corollary of Proposition \ref{prop-inner-collar-2}.

\begin{cor} \label{cor-collars}
	Given a triple $(\Sigma^{n-1}, g, H_o)$ where $H_o $ is a positive constant, suppose
	$g $ and $H_o$ satisfy 
	\be \label{eq-R-delta-H-c}
	\min_\Sigma R_g >  \frac{n-2}{n-1}H_o^{2} . 
	\ee
	Let $ \theta \in (0,1)  $ be the constant given   by 
	\be  \label{eq-def-theta-c}
	\theta=\frac{n-2}{n-1} \, \frac{H_o^2}{ \min_\Sigma R_g},
	\ee
	and define two constants $m$ and  $A_o$ by 
	\be\label{eq-m-defn-c}
	m=\frac{r_o^{n-2}}{2}\( 1-\theta \)  \ \  \mathrm{and} \ \ 
	A_o =\frac{n-1}{H_o r_o}\sqrt{1-\frac{2m}{r_o^{n-2}}},
	\ee
	respectively. 
	Then the metric
	\be 
	\gamma=A_o^2ds^2+r_o^{-2}v_m(s)^2g
	\ee
	defined on $T=[0,s_o]\times \S$ satisfies
	\begin{enumerate}[(i)]
		\item $R_\gamma \ge 0$; 
		\item the induced metric on $\S_0 = \{ 0 \} \times \Sigma $ by $\gamma$ is $g$;
		\item the mean curvature   of $\S_0$ with respect to $  - \p_s $  equals  $H_o$; 
		\item $\Sigma_s  = \{ s \} \times \Sigma $ 
		has positive  constant mean curvature with respect to $ - \p_s$ for each $ s \in [0, s_o ) $; and
		\item the other  boundary component $ \Sigma_{s_o} $ is a minimal surface 
		whose area  satisfies
		\be \label{eq-area-sigma-h-cor}
		\begin{split}
			\frac{1}{2}\( \frac{|\Sigma_{s_o}|}{\omega_{n-1}} \)^{\frac{n-2}{n-1}}
			=  & \  \frac{1}{2}\( \frac{|\Sigma|}{\omega_{n-1}} \)^{\frac{n-2}{n-1}}\( 1-\theta \). \\
		\end{split}
		\ee
	\end{enumerate}
	
\end{cor}

\section{Proof of  the  theorems} \label{sec-proof}

We are now in a position  to prove Theorems \ref{thm-main} and \ref{thm-main-Delta}. 
The idea  behind both proofs is the following: we attach an inner  collar extension, as constructed in Section \ref{sec-collars}, to the  given asymptotically flat manifold at its boundary, then apply the Riemannian Penrose inequality \eqref{eq-Penrose} to the resulting manifold.
While such a manifold  in general is not  smooth 
where the boundaries are joined, provided that the boundary mean curvature 
 from the inner side dominates that from the outer side  (cf. \cite{Miao02,ShiTam02}), 
it is known that the Riemannian Penrose inequality still applies. 
This is proven in \cite{Miao09} in the case $n=3$ and 
the same proof applies in dimensions $3<n\leq7$. 

\begin{prop}[Riemannian Penrose inequality on manifolds with corner along a hypersurface]\label{prop-RPIwC}
Let $ {\hM}^n$ denote a noncompact differentiable manifold of  dimension $ 3 \le n \le 7$, with compact boundary $\Sh$. 
Let  $ \Sigma$ be an embedded  hypersurface in the interior of $\hM$ such that $ \Sigma$ 
and $ \Sh$ bounds a bounded domain $ \Omega$. 
Suppose $ \gm$ is a $C^0$ metric on $\hM$ satisfying: 
\begin{itemize}
\item $\gm$ is smooth on  both  $ \hM \setminus \Omega$ and $\bOm= \Omega \cup \Sigma \cup \Sh$;
\item  $(\hM \setminus \Omega , g)$ is asymptotically flat; 
\item  $\gm$ has nonnegative scalar curvature away from $ \Sigma$; 
\item $H_- \ge H_+$, where $ H_-$ and $H_+$ denote the mean curvature of $ \Sigma$ in $(\bOm, \gm)$ and $ (\hM \setminus \Omega, \gm)$, respectively, with respect to the infinity-pointing normal;  and
\item $ \Sh $ is a minimal hypersurface in $(\bOm, \gm)$ and $ \Sh $ is outer minimizing in $ (\hM, \gm)$. 
\end{itemize}
Then the Riemannian Penrose inequality holds on $(\hM, \gm)$, i.e. 
\be  \label{eq-RPIwC}
	\m \geq \frac12\(\frac{|\Sh|}{\omega_{n-1}}\)^{\frac{n-2}{n-1}},
\ee 
where $\m$ is the ADM mass of $(\hM \setminus \Omega , \gm)$.
\end{prop}

We defer the proof  of Proposition \ref{prop-RPIwC} to Appendix \ref{app-RP} as
it is a repetition of the argument from  \cite{Miao09}, and now turn to proving the main theorems. 
We begin with  Theorem \ref{thm-main-Delta}.

\begin{proof}[Proof of Theorem \ref{thm-main-Delta}]
Given an asymptotically flat manifold  $M$ with boundary $ \Sigma$ satisfying the hypotheses of Theorem \ref{thm-main-Delta}, Proposition \ref{prop-inner-collar-2}  yields  a compact  manifold $(T,\gamma)$ with two boundary components, $ \Sigma_0 $ and $ \Sigma_{s_o}$. Since the induced metric from $\gamma$ on $ \Sigma_0 $, which is $ \{ 0 \} \times \Sigma $, equals 
the metric $g$ on $\Sigma$, 
we can attach   $(T, \gamma)$ to  $M$ 
by  matching the  Gaussian neighborhood of $\Sigma_0$ in $(T, \gamma)$ to 
that of $\Sigma$ in   $M$ along $\Sigma_0 = \Sigma$.
Denote the resulting manifold by  $\hat M$ and its metric by $\hat h$. 
By construction,  $\hat h$ is Lipschitz across $\Sigma$ and smooth everywhere else on $\hat M$;
it has nonnegative scalar curvature away from $\Sigma$; and the mean curvature of $ \Sigma$
from both sides in $\hat M $ agree.
Moreover,  $ \p \hat M = \Sigma_{s_o}$ is a minimal hypersurface that is 
outer minimizing in $ \hat M$. The outer minimizing property  is guaranteed by the 
fact that  $ \Sigma$ is outer minimizing in $M$ and the fact that $ (T, \gamma)$ is foliated by 
$ \Sigma_s = \{ s \} \times \Sigma$, $ s \in [0,s_o]$, 
which have positive mean curvature (with respect to $-\partial_s$, pointing towards infinity).
By Proposition \ref{prop-RPIwC}, the Riemannian Penrose inequality (Theorem \ref{thm-RPI}) applies to  such an  $\hat M$  to give
\be \label{eq-Penrose-app}
	\m \geq \frac12\( \frac{|\Sigma_{s_o} |}{\omega_{n-1}}  \)^{\frac{n-2}{n-1}}. 
\ee 
Theorem \ref{thm-main-Delta} now follows  from \eqref{eq-Penrose-app} and  \eqref{eq-area-sigma-h}. 
\end{proof}

Theorem \ref{thm-main} follows by an almost identical argument:

\begin{proof}[Proof of Theorem \ref{thm-main}]
If $H\equiv0$,  \eqref{eq-Penrose-like-1} is  the Riemannian Penrose inequality \eqref{eq-Penrose}.
So it suffices to assume $\max\limits_\Sigma H>0$. 
In this case,  let  $(T,\gamma)$ be the compact manifold given in
Corollary \ref{cor-collars} with the choice of $H_o = \max\limits_\Sigma H$. 
We attach $(T, \gamma)$ to $ M$ and  repeat the  previous proof. 
The only difference now is that the mean curvature of $ \Sigma$ in  $\hat M$ 
from the side of  $(T, \gamma)$ is $H_o$ while the mean curvature  from the side of $M$ is $H$.
By definition,  $ H_o \ge H $ everywhere on $ \Sigma$. 
By Proposition \ref{prop-RPIwC}, Theorem \ref{thm-RPI}  still holds on such an $\hat M$, we therefore have
\be \label{eq-Penrose-app-c}
	\m \geq \frac12\( \frac{|\Sigma_{s_o} |}{\omega_{n-1}}  \)^{\frac{n-2}{n-1}}. 
\ee 
Theorem \ref{thm-main} now follows  from \eqref{eq-Penrose-app-c} and  \eqref{eq-area-sigma-h-cor}. 
\end{proof}

In Theorems \ref{thm-main-Delta} and \ref{thm-main}, the connectedness of $ \Sigma$ 
is assumed only for the simplicity of the statement of the results. 
It is  clear  from the above proof that both theorems 
have analogues that allow the boundary $\p M$ to have multiple components. 
For instance, we have 

\begin{thm} \label{thm-main-multi-b}
Let $M^n$ be  an asymptotically flat  manifold of dimension  $3\leq n\leq 7$ with nonnegative scalar curvature, 
with outer minimizing boundary $\p M $ which has connected components $\Sigma_1$, $\ldots$, $\Sigma_k$. 
Let $ g_i $ be the induced metric on $ \Sigma_i $ and $H_i $ be the mean curvature of $\Sigma_i$, $ 1 \le i \le k$. 
For each $i$, suppose $g_i$ and $H_i$ satisfy 
\be \label{eq-condition-H-R-1-multi-b}
 \frac{n-2}{n-1} \max_\Sigma H_i^2 <  \min_\Sigma R_{g_i} , 
 \ee
 where $ R_{g_i}$ is  the scalar curvature of $(\Sigma_i, g_i)$. 
Then
\be \label{eq-Penrose-like-1-multi-b}
\begin{split}
\m \ge &   \ \frac12 \omega^{ - \frac{n-2}{n-1}} 
\( \sum_{i=1}^k  |  \Sigma_i |  ( 1 -  \theta_i )^{ - \frac{n-2}{n-1} }   \)^{\frac{n-2}{n-1}}  .
\end{split} 
\ee
Here  $ \m $ is the ADM mass of $M$,  $ | \Sigma_i | $ is the area of $ \Sigma_i $ and 
$ \theta_i = \frac{n-2}{n-1} \frac{ \max_\Sigma H_i^2 }{  \min_\Sigma R_{g_i}} $.  
\end{thm}

\begin{proof}
Let $ H_{o,i} = \max_{\Sigma_i} H_i $. 
If $ H_{o,i } = 0 $ for all $ i $, then \eqref{eq-Penrose-like-1-multi-b} follows from \eqref{eq-Penrose}. 
Without losing generality, we may assume 
$ H_{o,j} > 0 $, $ 1 \le j \le l $, for some $ 1 < l \le k $. 
For each such $j$,  let  $(T_j,\gamma_j)$ be the compact manifold given in
Corollary \ref{cor-collars} with the choice of $(\Sigma, g, H_o) = (\Sigma_j, g_j, H_{o,j} )$.  
We then attach each $(T_j, \gamma_j)$   to $ M$ at the corresponding $ \Sigma_j$ to obtain a manifold $\hat M$
and  proceed as in the proof of Theorem \ref{thm-main}.  
The boundary   $ \partial \hat M$  in this case consists of 
$ \tilde \Sigma_1, \ldots,  \tilde \Sigma_l , \Sigma_{l+1}, \ldots, \Sigma_k $, where $ \tilde \Sigma_j $ 
is the  boundary component of $ (T_j, \gamma_j)$ other than $\Sigma_j$.
The application of  Theorem \ref{thm-RPI} and \eqref{eq-area-sigma-h-cor}  then shows
\be
\begin{split}
\m \ge & \ \frac12 \(\frac{| \p \hat M  |}{\omega_{n-1}}\)^{\frac{n-2}{n-1}} \\
= & \   \frac12 \omega^{ - \frac{n-2}{n-1}} 
\( \sum_{j=1}^l | \tilde \Sigma_j | + \sum_{i=l+1}^k | \Sigma_k |  \)^{\frac{n-2}{n-1}}  \\
= & \   \frac12 \omega^{ - \frac{n-2}{n-1}} 
\( \sum_{i=1}^k  |  \Sigma_i |  ( 1 -  \theta_i )^{ - \frac{n-2}{n-1} }   \)^{\frac{n-2}{n-1}}  ,
\end{split} 
\ee
which proves \eqref{eq-Penrose-like-1-multi-b}. 
\end{proof}

In the above proof, 
replacing  the use of Corollary \ref{cor-collars} by  Proposition \ref{prop-inner-collar-2}, 
we obtain  the following analogue of Theorem \ref{thm-main-Delta} allowing  disconnected boundary. 

\begin{thm} \label{thm-main-Delta-multi-b}
Let $M^n$ be  an asymptotically flat  manifold of dimension  $3\leq n\leq 7$ with nonnegative scalar curvature, 
with outer minimizing boundary $\p M $ which has connected components $\Sigma_1$, $\ldots$, $\Sigma_k$. 
Let $ g_i $ be the induced metric on $ \Sigma_i $ and $H_i $ be the mean curvature of $\Sigma_i$, $ 1 \le i \le k$. 
For each $i$, suppose $g_i$ and $H_i$ satisfy 
\be \label{eq-R-delta-H0-m-b}
	R_{g_i} -2H_i\Delta H_i^{-1}-\frac{n-2}{n-1}H_i^2>0,
\ee 
where $ R_{g_i}$ is the scalar curvature of $ (\Sigma_i, g_i)$ and $ \Delta $ denotes the 
Laplace-Beltrami operator on $(\Sigma_i, g_i)$. 
Then
\be \label{eq-Penrose-like-1-Delta-multi-b}
\begin{split}
\m \ge &  \  \frac12 \omega^{ - \frac{n-2}{n-1}} 
\( \sum_{i=1}^k  |  \Sigma_i |  ( 1 -  \theta_i )^{ - \frac{n-2}{n-1} }   \)^{\frac{n-2}{n-1}}  .
\end{split} 
\ee
Here  $ \m $ is the ADM mass of $M$,  $ | \Sigma_i | $ is the area of $ \Sigma_i $ and 
$
\theta = \frac{n-2}{n-1}\,\max\limits_{\Sigma_i}\frac{H_i^2}{R_{g_i} -2H_i\Delta H_i^{-1}} .
$
\end{thm}

To end this paper, we comment  on  the equality case in Theorems \ref{thm-main} and 
\ref{thm-main-Delta}. By  the equality case in the Riemannian Penrose inequality (Theorem \ref{thm-RPI}), 
one  expects 
that equality in Theorems \ref{thm-main} and \ref{thm-main-Delta} hold only if $(M^n, g)$ is isometric to 
part of a spatial Schwarzschild manifold that lies outside a compact hypersurface 
 homologous to the 
horizon. In the context of the proof of Theorems \ref{thm-main} and \ref{thm-main-Delta}, this corresponds to 
showing that  equality in \eqref{eq-RPIwC}   would  imply   
  the manifold  $(\hM, \gm)$ in Proposition \ref{prop-RPIwC}  is isometric to a spatial Schwarzschild manifold that lies outside its horizon.
Since our current proof of Proposition \ref{prop-RPIwC} (see Appendix \ref{app-RP} below)  relies on approximating
$(\hM, \gm) $ by a sequence of smooth manifolds, and  equality on $(\hM, \gm) $ does not 
necessarily  translates into  equality for elements in the approximating sequence,  we do not have  a rigidity statement 
in  Proposition \ref{prop-RPIwC}.

\begin{appendices}

\makeatletter
\def\@seccntformat#1{Appendix\ \csname the#1\endcsname\quad}
\makeatother

\section{ } \label{app-RP}

In this appendix we give a proof of Proposition \ref{prop-RPIwC}, which 
was essentially proved in \cite{Miao09} when $n=3$.
The following proof follows closely  arguments  on pages 279 - 280 in \cite{Miao09}, 
starting from the proof of Lemma 4 therein and ending at its  equation (47).

\begin{proof}[Proof of Proposition \ref{prop-RPIwC}]
We will apply the approximation scheme in \cite{Miao02} to the doubling of $(\hM, \gm)$ across its boundary $ \Sh$.
To be precise, 
let  $ (\hM_c, \gm_c)$ be a copy of $(\hM, \gm)$. 
We glue $(\hM, \gm) $ and $ (\hM_c, \gm_c)$ along their common boundary $\Sh$ 
to form a  Riemannian manifold $(\thM, \tgm) $ that has two asymptotically flat ends. 
Let $ \Sigma_c$ denote the copy of $ \Sigma$ in $\hM_c$. 
The metric $\tgm $ on $ \thM$  is smooth away from $ S  =  \Sigma \cup  \Sh \cup  \Sigma_c$.
Because of  the condition $H_- \ge H_+$ at $ \Sigma $  and because of  the fact that 
 $ \Sh$ is minimal  in $(\hM, \gm)$, we can 
 apply Proposition 3.1 in \cite{Miao02} to $(\thM, \tgm)$ at $S$, 
 followed by the  conformal deformation specified  in \cite[Section 4.1]{Miao02}, 
to  get a sequence of smooth asymptotically flat metrics $ \{\th_k \}$ on $\thM$ such that 
\begin{itemize}
\item $ \th_k $ has nonnegative scalar curvature;
\item  $ \{ \th_k \}$ converges uniformly to $\tgm$ in the $C^0$-topology; and 
\item  the mass $\m_k$ of $\th_k$ converges to the  mass $\m$ of  $\tgm $ on each end of $\thM$.
\end{itemize} 
Furthermore, as $(\thM, \tgm)$ has  a reflection isometry 
 (which maps a point $ x \in \hM$  to its copy in $\hM_c$), 
 a careful inspection of the analysis in \cite[Section 3]{Miao02} reveals that the metrics $ \{\th_k\} $ can be chosen so as to preserve 
this  reflection symmetry. (Specifically, this is ensured by choosing the mollifier $\phi(t)$ and cut-off function $\sigma(t)$, in (8) and (9) of \cite{Miao02} respectively, to be even functions.)  As a result, $\Sh$ is a totally geodesic, hence minimal,
hypersurface in  $(\thM, \tilde h_k)$ for each $k$.
 
 At this point we turn our attention to the restriction of $\tilde h_k$ to $\hM$, which we denote by $h_k$.
  Given any compact hypersurface $ \Sigma' \subset \hM$,  let 
$|\Sigma'|_k$ and $ | \Sigma' | $  denote the area of $\Sigma'$ computed in $(\hM, h_k)$ and $ (\hM, \gm)$, respectively. 
Let $\mathcal{S}$ be  the set of all closed  hypersurfaces $\Sigma'$  in $\hM$ which  enclose $\Sigma_H$.
 Since $ \Sh$ is minimal in $(\hM, h_k)$ and the dimension $n$ satisfies $ 3 \le n \le 7$,
 there exists an element $ \Sigma_k \in \mathcal{S} $  such that 
  \bee
 | \Sigma_k |_k = \inf_{ \Sigma' \in \mathcal{S} } | \Sigma' |_k ,
 \eee
 and hence $\Sigma_k$ is minimal in $(\hM, h_k)$.
Moreover, as  $ | \Sigma_k |_k \le | \Sh |_k $, we have   
  \be \label{limit-1}
   \limsup_{k \to \infty} | \Sigma_k |_k  \le \lim_{ k \to \infty}  | \Sh |_k = | \Sh | . 
  \ee
 In particular,  $  | \Sigma_k |_k  \le C $ for some $ C> 0$ independent on $k$. 
 This together with the fact that  $\{ h_k \}$ converges uniformly  to $\gm$ in $C^0$-topology then implies  
 \bee
 \label{eq-lim-Sigma}	\lim_{k \rightarrow\infty} ( |\Sigma_k|_k  - | \Sigma_k| ) = 0 .
 \eee
 Consequently, 
 \be \label{limit-2}
  \liminf_{k \to \infty} | \Sigma_k |_k  =  \liminf_{ k \to \infty}  | \Sigma_k | \ge  | \Sh | ,
 \ee
 where we have  used the fact $ \Sh $ is outer minimizing in $(\hM, \gm)$ to obtain  the inequality. 
 Thus, it follows from \eqref{limit-1} and \eqref{limit-2}  that
 \be \label{limit}
  \lim_{k \to \infty} | \Sigma_k |_k  =   | \Sh | . 
 \ee
 To finish the proof,  let $ \hM_k \subset \hM $ denote the region that lies outside $\Sigma_k$. 
Theorem \ref{thm-RPI} applies to $(\hM_k, h_k)$ to  give
 \be \label{eq-RPI-s}
 \m_k \ge \frac12\(\frac{|\Sigma_k|_k}{\omega_{n-1}}\)^{\frac{n-2}{n-1}} . 
 \ee
Passing to the limit,  \eqref{eq-RPIwC}  follows  from \eqref{limit}, \eqref{eq-RPI-s} and the fact
$ \lim_{ k \to \infty} \m_k = \m $. 
 \end{proof}

\end{appendices}

\end{document}